\documentclass[10pt]{amsart}

\usepackage{amsmath, amsthm, amscd, amssymb, amsfonts, pstricks,bbm, verbatim}

\input amssym.def
\input amssym.tex

\newcommand{\Rb}{{\mathbb R}}

\newcommand{\1}{\mathbf{1}}
\def\ip<#1,#2>{\langle #1, #2 \rangle}

\newtheorem{thm}{Theorem}[section]
\newtheorem*{thm*}{Theorem}
\newtheorem{cor}[thm]{Corollary}
\newtheorem*{cor*}{Corollary}
\newtheorem{lemma}[thm]{Lemma}

\newtheorem*{con*}{Conjecture}
\newtheorem*{prob*}{Problem}

\theoremstyle{definition}
\newtheorem{defn}[thm]{Definition}

\theoremstyle{remark}
\newtheorem{rem}[thm]{Remark}

\begin{document}
\title{Distance matrices of subsets of the Hamming cube}
\author[Ian Doust]{Ian Doust}
\email{i.doust@unsw.edu.au}
\address{School of Mathematics and Statistics, University of New South Wales, Sydney, New South Wales 2052, Australia}
\author[Gavin Robertson]{Gavin Robertson}
\email{gavin.robertson@unsw.edu.au}
\address{School of Mathematics and Statistics, University of New South Wales, Sydney, New South Wales 2052, Australia}
\author[Alan Stoneham]{Alan Stoneham}
\email{a.stoneham@unsw.edu.au}
\address{School of Mathematics and Statistics, University of New South Wales, Sydney, New South Wales 2052, Australia}
\author[Anthony Weston]{Anthony Weston}
\email{anthony.weston@qatar.tamu.edu}
\address{Texas A\&M University at Qatar, Science Program, PO Box 23874, Education City, Doha, Qatar}
\address{Department of Decision Sciences, University of South Africa, PO Box 392, UNISA 0003, South Africa}



\begin{abstract}
Graham and Winkler derived a formula for the determinant of the distance matrix
of a full-dimensional set of $n + 1$ points $\{ x_{0}, x_{1}, \ldots , x_{n} \}$
in the Hamming cube $H_{n} = ( \{ 0,1 \}^{n}, \ell_{1} )$. In this article we
derive a formula for the determinant of the distance matrix $D$ of an arbitrary set
of $m + 1$ points $\{ x_{0}, x_{1}, \ldots , x_{m} \}$ in $H_{n}$. It follows
from this more general formula that $\det (D) \not= 0$ if and only if the vectors
$x_{0}, x_{1}, \ldots , x_{m}$ are affinely independent. Specializing to the case
$m = n$ provides new insights into the original formula of Graham and Winkler.
A significant difference that arises between the cases $m < n$ and $m = n$ is noted.
We also show that if $D$ is the distance matrix of an unweighted tree
on $n + 1$ vertices, then $\ip< D^{-1} \1,\1> = 2/n$ where $\1$
is the column vector all of whose coordinates are $1$. Finally, we derive
a new proof of Murugan's classification of the subsets of $H_{n}$ that
have strict $1$-negative type.
\end{abstract}
\maketitle

\section{Introduction: Distance and Gram matrices}\label{Sec 1}
The global geometry of a finite metric space $(\{x_{0}, x_{1}, \ldots , x_{m} \}, d)$
is completely encoded within its \textit{distance matrix} $D = (d(x_{i}, x_{j}))_{i,j = 0}^{m}$.
The distance matrices we focus on in this article correspond to metric subspaces of the
\textit{Hamming cube} $H_{n} = (  \{ 0,1 \}^{n}, \ell_{1} )$. In this context,
the metric distance $d$ between two vectors $x, y \in \{ 0,1 \}^{n}$ is given by
$d( x,y ) = \| x - y \|_{1}$. Associated with such a distance matrix $D$
is a Gram matrix $G = G(D)$ that will be described in Section \ref{Sec 2}.
It is, however, helpful at this point to recall some general properties
of Gram matrices.

The \textit{Gram matrix} of a set of vectors $\{ x_{1}, \ldots , x_{m} \} \subset
\mathbb{R}^{n}$ is the $m \times m$ matrix
\begin{align*}
G(x_{1}, \ldots , x_{m}) & = (x_{i} \cdot x_{j})_{i,j = 1}^{m} \\
                         & = BB^{T},
\end{align*}
where $B$ is the $m \times n$ matrix whose $i$th row is given by the vector
$x_{i}$, $1 \leq i \leq m$. One may use the Gram matrix $G(x_{1}, \ldots , x_{m})$
to test for linear dependence. Indeed, the set of vectors $\{ x_{1}, \ldots , x_{m} \}$
is linearly dependent if and only if $\det G(x_{1}, \ldots , x_{m}) = 0$. This
result is known as Gram's criterion for linear dependence. All Gram matrices are
positive semi-definite. Moreover, $G(x_{1}, \ldots , x_{m})$ is positive definite
if and only if the set of vectors $\{ x_{1}, \ldots , x_{m} \}$ is linearly
independent. For a classical treatment of these results, see Gantmacher \cite{FG}.

Gram matrices also arise naturally when calculating the volumes of $m$-dimensional
parallelepipeds in $\mathbb{R}^{n}$. Given linearly independent vectors $x_{1}, x_{2}, \ldots, x_{m} \in \mathbb{R}^{n}$,
the $m$-dimensional parallelepiped with sides $x_{1}, x_{2}, \ldots, x_{m}$ is, by definition,
the set
\begin{align*}
P = \{ t_{1}x_{1} + t_{2}x_{2} + \cdots + t_{m}x_{m} \, | \, t_{k} \in [0,1], 1 \leq k \leq m \}.
\end{align*}
For any given set of vectors $\{ x_{1}, \ldots , x_{m} \} \subset \mathbb{R}^{n}$,
the $m$-dimensional volume $V$ of the parallelepiped $P$ with sides $x_{1}, \ldots , x_{m}$ satisfies
$V^{2} = \det G(x_{1}, \ldots , x_{m})$. See Courant and John \cite{CJ} for a
comprehensive treatment of volumes of parallelepideds.

A set of points $S$ in the Hamming cube $H_{n}$
is said to be \textit{full-dimensional} if the convex hull of $S$ has positive
$n$-dimensional volume. Notably, a set $S$ of $n + 1$ points in $H_{n}$ is
full-dimensional if and only if $S$ is an affinely independent subset of $\mathbb{R}^{n}$.

As an example, it is well-known that every $n + 1$ point metric tree $T$ endowed
with the usual graph metric $\rho$ embeds isometrically into $H_{n}$. So it
follows from results of Hjorth et al.\ \cite{Hj1} and Murugan \cite{MM} that
the embedded vertices of $T$ form a full-dimensional subset of $H_{n}$. The
determinant of the distance matrix $D$ of any such metric tree $(T, \rho)$ is
given by $\det (D) = (-1)^{n}n2^{n - 1}$ and hence does not depend upon the
geometry of the particular tree $T$. This remarkable formula is due to Graham
and Pollak \cite{GP}. Graham and Winkler \cite{GW, GW2} generalized this tree result
by calculating the determinant of the distance matrix $D$ of any full-dimensional
set of $n+1$ points $\{ x_{0}, x_{1}, \ldots , x_{n} \}$ in $H_{n}$. They showed that
\begin{align}\label{twinkle}
\det (D) & = (-1)^{n}n2^{n-1} \det G(x_{1} - x_{0}, \ldots , x_{n} - x_{0}) \nonumber \\
         & = (-1)^{n}n2^{n-1} V^{2},
\end{align}
where $V$ is the volume of the parallelepiped with sides
$x_{1} - x_{0}, \ldots , x_{n} - x_{0}$.

In this article we calculate the determinant of the distance matrix $D$ of an
arbitrary set $\{ x_{0}, x_{1}, \ldots , x_{m} \} \subseteq H_{n}$, $m \geq 1$. These calculations are implemented in Section \ref{Sec 2}.
The resulting formulas are stated in Lemma \ref{Lemma 1}, Theorem \ref{Thm 1} and Theorem
\ref{Thm 2}. It follows that $\det (D) \not= 0$ if and only if the set of vectors
$\{ x_{0}, x_{1}, \ldots , x_{m} \}$ is affinely independent.
In the case $m = n$ we show how to reduce to the aforementioned result (\ref{twinkle})
of Graham and Winkler \cite{GW, GW2}. In Remark \ref{Rem 1}, we point out that there is
a significant difference between the cases $m < n$ and $m = n$.

In Section \ref{Sec 3} we use the formulas from Section \ref{Sec 2} and a theorem of S\'{a}nchez \cite{SS} to
provide a new proof of Murugan's \cite{MM} classification of the subsets of $H_{n}$ that have strict $1$-negative type.
Given the distance matrix $D$ of an affinely independent set $ \{ x_{0}, x_{1}, \ldots , x_{m} \} \subset H_{n}$,
it becomes necessary to calculate a formula for the inner product $\langle D^{-1}\1,\1\rangle$, where
$\1$ is the column $(m+1)$-vector all of whose coordinates are $1$. This is done in Corollary \ref{Cor M}.

In Section \ref{Sec 4} we consider the case of an embedded $n+1$ point unweighted metric tree in $H_{n}$
and show that $\langle D^{-1}\1,\1\rangle = 2/n$. We conjecture that this quantity is, in fact, minimal
over all affinely independent subsets $\{ x_{0}, x_{1}, \ldots , x_{m} \}$ of $H_{n}$ and provide some numerical
evidence.

Throughout we assume that $n \geq 2$ is a fixed integer and let $H_{n}$ denote the $n$-dimensional hypercube
$\{ 0,1 \}^{n}$ endowed with the $\ell_{1}$-metric.

\section{Distance matrices of subsets of the Hamming cube}\label{Sec 2}

The hypercube $\{0,1\}^n$ has a natural additive group structure given by elementwise addition modulo 2.
It is easy to check that the $\ell_1$ metric is invariant under translation in this group. In particular,
if $X = \{x_0,x_1,\dots,x_m\}$ is a given subset of $H_{n}$, then the map $\Phi(x) = x-x_0$ is an isometric
isomorphism from $X$ to the set $X' = \{{\bf 0}, x_1-x_0,\dots,x_m-x_0\}$ (where $\mathbf{0}$ denotes
the zero vector in $\{ 0,1 \}^{n}$). Note that
if we set $x_i' = x_i - x_0$, $0 \le i \le m$, then obviously $G(x_1-x_0,\dots,x_m-x_0) = G(x_1' - x_0',\dots,x_m' - x_0')$
and consequently the $m$-dimensional parallelepided with edges $x_{1} - x_{0}, \ldots, x_{m} - x_{0}$ has
the same volume as the $m$-dimensional parallelepided with edges $x_{1}^{\prime}, \ldots, x_{m}^{\prime}$.

Consequently, when considering a set $\{ x_{0}, x_{1}, \ldots, x_{m} \}\subseteq H_n$
we may assume that $x_{0} = \mathbf{0}$ without altering the distance matrix $D$ and without
altering the volume of the $m$-dimensional parallelepided with edges $x_{1} - x_{0}, \ldots, x_{m} - x_{0}$.

Henceforth, given a set $\{ x_{0}, x_{1}, \ldots, x_{m} \} \subseteq H_{n}$,
we will assume that $x_{0} = \mathbf{0}$ unless stated otherwise.
Throughout the distance matrix of $\{ x_{0}, x_{1}, \ldots, x_{m} \}$ will be denoted by
$D = (d(x_{i}, x_{j}))_{i,j = 0}^{m} = ( \| x_{i} - x_{j} \|_{1})_{i,j = 0}^{m}$. We associate with $D$ the
$m \times n$ matrix $B$ whose $i$th row is given by $x_{i}$, $1 \leq i \leq m$. So if
$x_{i} = (x_{i,1}, x_{i,2}, \ldots, x_{i,n})$, then $B = (x_{i,j})_{i = 1,j = 1}^{m,n}$.
The matrix product $BB^{T}$ is the $m \times m$ Gram matrix $G = G(x_{1}, \ldots, x_{m}) =
(x_{i} \cdot x_{j})_{i,j = 1}^{m}$. It it also useful to let $u\in\mathbb{R}^{m}$ denote the column
vector $u=(x_{1} \cdot x_{1}, x_{2} \cdot x_{2}, \ldots, x_{m} \cdot x_{m})^T$. Finally, we will use
$\langle\cdot,\cdot\rangle$ to denote the standard inner product on $\Rb^{n}$ when dealing with
quadratic forms such as $(G^{-1}u)\cdot u=u^{T}G^{-1}u=\langle G^{-1}u,u\rangle$. With these notational
preliminaries in mind it is instructive to compare $\det (D)$ to $\det (G)$.

\begin{lemma}\label{Lemma 1}
Let $\{ x_{0}, x_{1}, \ldots, x_{m} \}$, $m \geq 1$, be a subset of the Hamming
cube $H_{n}$. Then
\[
\det(D) = (-1)^{m-1}2^{m-1} \det \begin{pmatrix} 0 & u^{T} \\ u & G \end{pmatrix}.
\]
\end{lemma}

\begin{proof}
The following simple identity will be used. For $x, y \in H_{n}$,
\begin{align}\label{eqn1}
d(x, y)   = (x - y) \cdot (x - y)  
          = (x \cdot x) + (y \cdot y) - 2(x \cdot y).
\end{align}
Let $R_{i}$ and $C_{j}$ denote the $i$-th row and $j$-th column of $D$ respectively.
Consider the matrix $D^{\prime}$ that is obtained by applying the following elementary
row and column operations to $D$. For $2 \leq i,j \leq m+1$ replace $R_{i}$ by $R_{i}-R_{1}$
and then replace $C_{j}$ by $C_{j}-C_{1}$. Using (\ref{eqn1}) we see that
\[D^{\prime}=\begin{pmatrix} 0 & u^{T} \\ u & -2G \end{pmatrix}\]
and so
\[\det(D) = \det \begin{pmatrix} 0 & u^{T} \\ u & -2G \end{pmatrix}.\]
The result now follows by factorizing $-2$ from each of the entries of $D^{\prime}$ and
replacing such a factor for the first row and the first column.
\end{proof}

\begin{thm}\label{Thm 1}
Let $\{ x_{0}, x_{1}, \ldots, x_{m} \}$, $m \geq 1$, be a subset of the Hamming
cube $H_{n}$. If the set of vectors $\{ x_{1}, x_{2}, \dots, x_{m} \}$ is linearly dependent then $\det(D) = 0$.
Consequently,
\[
\det \begin{pmatrix} 0 & u^{T} \\ u & G \end{pmatrix} = 0.
\]
\end{thm}

\begin{proof}
Suppose that  $\{ x_{1}, x_{2}, \dots, x_{m} \}$ is linearly dependent subset of $H_{n}$.
Then there exist scalars $c_{1}, \dots ,c_{m}\in\mathbb{R}$, not all zero, such that $\sum_{j=1}^{m} c_{j}x_{j} = \mathbf{0}$.
Now set $c_{0} =- \sum_{j=1}^{m} c_{j}$ and $c =(c_{0}, c_{1}, \ldots,c_{m})^T$. Then for each $i = 0, \ldots ,n$,
\begin{align*}
(Dc)_{i}        & =  \sum_{j=0}^{m} d(x_{i}, x_{j})c_{j} \\
                & =  d(x_{i}, x_{0})c_{0} + \sum_{j=1}^{m} \left( \sum_{k=1}^n|x_{i,k} - x_{j,k}| \right)c_{j}.
\end{align*}
Note that since $x_{i,k}$, $x_{j,k}$ and $|x_{i,k} - x_{j,k}|$ are all either $0$ or $1$, we have that
\begin{align*}
|x_{i,k} - x_{j,k}| & = (x_{i,k} - x_{j,k})^{2} = x_{i,k} - 2x_{i,k}x_{j,k} + x_{j,k}.
\end{align*}
Recalling that $\sum_{j=1}^{m} c_{j}x_{j} = \mathbf{0}$ and swapping the order of summation yields
\begin{align*}
(Dc)_{i}        & = c_{0} \sum_{k=1}^{n} x_{i,k} +
                      \sum_{k=1}^{n} \left( \sum_{j=1}^{m} \left(x_{i,k} - 2x_{i,k}x_{j,k} + x_{j,k}\right)c_{j} \right) \\
                & = c_{0} \sum_{k=1}^{n} x_{i,k} + \sum_{k=1}^{n}\sum_{j=1}^{m}x_{i,k}c_{j} + \sum_{k=1}^{n}(1-2x_{i,k})\bigg(\sum_{j=1}^{m}c_{j}x_{j}\bigg)_{k}\\
                & = c_{0} \sum_{k=1}^{n} x_{i,k} + \left(\sum_{j=1}^{m} c_{j}\right)\sum_{k=1}^{n}x_{i,k}\\
                & = (c_{0} - c_{0})\sum_{k=1}^{n} x_{i,k} \\
                & = 0.
\end{align*}
Hence $c$ is a nonzero element of the kernel of $D$, and so we conclude that $\det (D)=0$. The fact that
\[
\det \begin{pmatrix} 0 & u^{T} \\ u & G \end{pmatrix} = 0
\]
now follows immediately from this and Lemma \ref{Lemma 1}.
\end{proof}

\begin{lemma}\label{Lemma 2}
Suppose that $W$, $X$, $Y$ and $Z$ are matrices of sizes $j \times j$, $j \times k$, $k \times j$
and $k \times k$ (respectively) and that $Z$ is invertible. Then
\[\det\begin{pmatrix} W & X \\ Y & Z \end{pmatrix}=\det(Z)\det(W-XZ^{-1}Y).\]
\end{lemma}

\begin{proof}
Simply note the factorization
\[\begin{pmatrix} W & X \\ Y & Z \end{pmatrix}=\begin{pmatrix} I & X \\
0 & Z \end{pmatrix}\begin{pmatrix} W-XZ^{-1}Y & 0 \\ \
  Z^{-1}Y & I \end{pmatrix},\]
and take the determinant of both sides.
\end{proof}

\begin{thm}\label{Thm 2}
Let $\{ x_{0}, x_{1}, \ldots, x_{m} \}$, $m \geq 1$, be a subset of the Hamming
cube $H_{n}$. If the set of vectors $\{ x_{1}, x_{2}, \dots, x_{m} \}$ is linearly independent, then
\begin{align*}
\det (D) & = (-1)^{m}2^{m-1}\det(G)\langle G^{-1}u,u\rangle \\
         & = (-1)^{m}2^{m-1}V^{2}\langle G^{-1}u,u\rangle,
\end{align*}
where $V$ is the volume of the $m$-dimensional parallepiped with sides
$x_{1}, x_{2}, \dots, x_{m}$. In particular, $\det (D) \not= 0$.
\end{thm}

\begin{proof}
The Gram matrix $G$ is invertible (and thus has a non-zero determinant)
because the vectors $x_{1}, x_{2}, \dots, x_{m}$ are linearly independent.
Hence the formulas for $\det (D)$ follow immediately from
Lemmas \ref{Lemma 1} and \ref{Lemma 2}. Moreover, because
the Gram matrix $G$ is positive definite, $\langle G^{-1}u,u\rangle \not= 0$.
Consequently, $\det (D) \not= 0$.
\end{proof}

The following corollary holds for any set of vectors $\{ x_{0}, x_{1}, \ldots, x_{m} \} \subseteq H_{n}$.
In particular, it may be the case that $x_{0} \not= \mathbf{0}$.

\begin{cor}\label{affine}
Let $\{ x_{0}, x_{1}, \ldots, x_{m} \}$, $m \geq 1$, be a subset of the Hamming
cube $H_{n}$. Then $\det (D) \not= 0$ if and
only if the set of vectors $\{ x_{0}, x_{1}, \ldots, x_{m} \}$ is affinely independent.
\end{cor}

\begin{proof}
This is an immediate consequence of Theorems \ref{Thm 1} and \ref{Thm 2}.
\end{proof}

We now reduce to the result (\ref{twinkle}) of Graham and Winkler \cite{GW, GW2}
stated in Section \ref{Sec 1} by calculating $\langle G^{-1}u,u\rangle$ in the case $m = n$.

\begin{thm}\label{Thm 3}
Let $\{ x_{0}, x_{1}, \ldots, x_{n} \}$ be a subset of the Hamming
cube $H_{n}$. If the set of vectors $\{ x_{1}, x_{2}, \dots, x_{n} \}$ is linearly independent,
then $\langle G^{-1}u,u\rangle = n$.
\end{thm}

\begin{proof}
Let $\1$ denote the vector in $\Rb^{n}$ all of whose coordinates are $1$. It is easy to verify that
$B\1=u$. In this setting ($m=n$) we have the advantage that the matrix $B$ is invertible and hence $B^{-1}u=\1$. Then we simply calculate that
	\begin{align*}
	\langle G^{-1}u,u\rangle&=\langle(BB^{T})^{-1}u,u\rangle
	\\&=\langle(B^{-1})^{T}B^{-1}u,u\rangle
	\\&=\langle B^{-1}u,B^{-1}u\rangle
	\\&=\langle\1,\1\rangle
	\\&=n.
	\end{align*}
\end{proof}

\begin{rem}\label{Rem 1}
Let $\{ x_{0}, x_{1}, \ldots, x_{m} \}$, $m \geq 1$, be a subset of the Hamming
cube $H_{n}$. It is worth noting that if the set of vectors $\{ x_{1}, x_{2}, \dots, x_{m} \}$ is linearly independent and $m < n$,
then it need not be the case that $\langle G^{-1}u,u\rangle = m$. If, for instance, we consider two linearly independent
vectors $x_{1}, x_{2} \in \{ 0,1 \}^{n}$, then direct calculations show that
\begin{align}\label{eqn3}
\langle G^{-1}u,u\rangle = \frac{(x_{1} \cdot x_{1})(x_{2} \cdot x_{2})(x_{1} - x_{2}) \cdot (x_{1} - x_{2})}{\det (G)}.
\end{align}
If $n = 2$ the quantity on the right side of (\ref{eqn3}) is easily seen to be equal to $2$ and this is consistent
with Theorem \ref{Thm 3}. But, in general, if $n > 2$ the quantity on the right side of (\ref{eqn3}) is not even
constant. To see that this is so it suffices to consider the case $n = 3$. If $x_{1} = (1,1,1)$ and $x_{2}= (1,1,0)$,
then $\langle G^{-1}u,u\rangle = 3$. On the other hand, if $x_{1} = (1,0,1)$ and $x_{2}= (1,1,0)$, then
$\langle G^{-1}u,u\rangle = 8/3$. In general, the calculation of $\langle G^{-1}u,u\rangle$ is more nuanced in the case $m < n$
because $B$ is no longer invertible.
\end{rem}

The following corollary holds for any set of vectors $\{ x_{0}, x_{1}, \ldots, x_{n} \} \subseteq H_{n}$.
In particular, it may be the case that $x_{0} \not= \mathbf{0}$.

\begin{cor}
Let $\{ x_{0}, x_{1}, \ldots, x_{n} \}$ be a subset of the Hamming
cube $H_{n}$. If the set of vectors $\{ x_{0}, x_{1}, \dots, x_{n} \}$ is affinely independent, then
$\det (D) = (-1)^{n}n2^{n-1}V^{2}$,
where $V$ is the volume of the $n$-dimensional parallepiped with sides
$x_{j} - x_{0}$, $1 \leq j \leq n$.
\end{cor}

\begin{proof}
Immediate from Theorems \ref{Thm 2} and \ref{Thm 3}.
\end{proof}

In the event that an affinely independent set $\{ x_{0}, x_{1}, \dots, x_{n} \} \subset H_{n}$ is an embedded
$n + 1$ point unweighted tree we have that $\det (D) = (-1)^{n}n2^{n-1}$ by the celebrated formula of Graham and Pollak \cite{GP}.
It is worth noting that there exist affinely independent sets $\{ x_{0}, x_{1}, \dots, x_{n} \} \subset H_{n}$ that
satisfy this same formula but which are not embedded trees. For instance, if
we set $x_{0} = \mathbf{0}$, $x_{1} = (1,0,0)$, $x_{2} = (0,1,0)$ and $x_{3} = (1,1,1)$ in $H_{3}$, then it is easy to
verify that $\det(D) = -12$. However, $\{ x_{0}, x_{1}, x_{2}, x_{3} \}$ is certainly not an embedded tree in $H_{3}$.
  
\section{Applications to supremal negative type}\label{Sec 3}

The results of Section \ref{Sec 2} afford a new analysis of negative type properties of subsets of the
Hamming cube $H_{n}$. In particular, we develop a new proof of Murugan's \cite{MM} classification of
the subsets of $H_{n}$ that have strict $1$-negative type. In order to proceed we need to recall some
classical definitions and related theorems.

\begin{defn}
  Let $(X, d)$ be a metric space and suppose that $p \geq 0$. Then:
  \begin{enumerate}
  \item[(a)] $(X, d)$ has \textit{$p$-negative type} iff for each finite subset $\{ x_{0}, \ldots, x_{m} \}$ of $X$
     and each choice of scalars $\xi_{0}, \ldots, \xi_{m}$ such that $\xi_{0} + \cdots + \xi_{m} = 0$, we have
    \begin{eqnarray}\label{neg type}
    \sum_{i,j=0}^{m} d(x_{i}, x_{j})^{p} \xi_{i} \xi_{j} & \le & 0.  
    \end{eqnarray}  
    
  \item[(b)] $(X, d)$ has \textit{strict $p$-negative type} iff  $(X, d)$ has \textit{$p$-negative type} and, moreover,
    each inequality (\ref{neg type}) is strict whenever $(\xi_{1}, \ldots, \xi_{n}) \not= \mathbf{0}$.

  \item[(c)] The \textit{supremal negative type} of $(X, d)$, denoted by $\wp_{(X,d)}$ or simply $\wp_{X}$ when the metric
    $d$ is clear, is defined to be the supremum of all $p \geq 0$ such that $(X, d)$ has $p$-negative type.
  \end{enumerate}
\end{defn}

Considerations of negative type arose classically in relation to fundamental isometric embedding problems.
For instance, Schoenberg famously determined that a metric space $(X, d)$ embeds isometrically into
a Hilbert space $H$ iff $(X, d)$ has $2$-negative type. Moreover, the range of the embedding will be an
affinely independent subset of $H$ iff $(X, d)$ has strict $2$-negative type. Schoenberg further determined that
if a metric space $(X, d)$ has $p$-negative type then it has $q$-negative type for all $q \leq p$
and that $\wp_{(X, d)}$ is a maximum whenever it is finite. These results appear in Schoenberg \cite{IS1, IS2, IS3}.

Subsets of $\ell_{1}$ are well-known to have to have $1$-negative type. (See, for instance, Wells and Williams
\cite[Theorem 4.10]{WW}.) Explicitly determining subsets of $\ell_{1}$
that have strict $1$-negative type is a significantly more challenging (and largely open) problem. A nice
result in this direction is the following theorem of Murugan \cite{MM}: A subset $X = \{ x_{0}, x_{1}, \ldots, x_{m}\}$ of the Hamming
cube $H_{n}$ has supremal negative type $\wp_{X} = 1$ iff the set of vectors
$\{ x_{0}, x_{1}, \ldots, x_{m}\}$ is affinely dependent. Equivalently, since the supremal negative type of a finite
metric space cannot be strict by Li and Weston \cite{LW}, it follows that $\wp_{X} > 1$ iff the set of vectors
$\{ x_{0}, x_{1}, \ldots, x_{m}\}$ is affinely independent. Stated this way, we see that there is a strong correlation
between Murugan's theorem and Corollary \ref{affine}.

For any metric space $(X, d)$ and any $\alpha \in (0, 1)$, the so-called \textit{metric transform} $d^{\alpha}$ is also
a metric on $X$ and it is easy to verify that $\wp_{(X, d^\alpha)}  = \alpha^{-1}\wp_{(X,d)}$. Now, for $p \geq 1$, let $d_{p}$ denote
the $\ell_{p}$-metric on $\mathbb{R}^{n}$. For any $x, y \in \{ 0, 1\}^{n}$ it is plain to see that $d_{p}(x, y) = d_{1}(x, y)^{1/p}$.
So, for any $X \subseteq  \{ 0, 1\}^{n}$, it follows that $\wp_{(X,d_{p})} = p \wp_{(X,d_{1})}$. As $\wp_{(X, d_{1})} \geq 1$, we deduce
that $\wp_{(X, d_{p})} \geq p$ for all $p \geq 1$. It is also worth noting that in the case $p = \infty$,
$d_{\infty}$ is necessarily the discrete metric on $X$, and so $\wp_{(X, d_{\infty})} = \infty$. 

In general, given a metric space $(X, d)$, explicitly calculating or even estimating $\wp_{(X, d)}$ is a difficult
exercise in combinatorial optimization. In the case of a finite metric space $(X, d) = (\{ x_{0}, \ldots x_{m} \}, d)$,
S{\'a}nchez \cite{SS} gave an explicit formula for $\wp_{(X,d)}$ in terms of the underlying $p$-distance matrices
$D_{p} = (d(x_{i}, x_{j})^{p})_{i, j = 0}^{m}$, $p \geq 0$. Namely,
\begin{equation}\label{eqn4}
\wp_{(X,d)} = \min \{ p \,:\, \text{$\det(D_p) = 0$ or $\ip<D_p^{-1}\1,\1> = 0$} \}
\end{equation}
where $\1$ is the column vector all of whose coordinates are $1$. In particular,
this shows that if $\det(D_p) = 0$ then $\wp_{(X,d)} \le p$.

S{\'a}nchez' proof of (\ref{eqn4}) depends upon the following theorem.

\begin{thm}[S{\'a}nchez \cite{SS}]\label{Thm 3.5}
  Let $|X| > 1$ and $(X, d)$ be a finite metric space of $p$-negative type. Then $(X, d)$
  has strict $p$-negative type iff
  \begin{enumerate}
  \item $\det (D_{p}) \not= 0$, and

  \item $\ip<D_p^{-1}\1,\1> \not= 0$.
  \end{enumerate}  
\end{thm}

Now consider a set $X = \{ x_{0}, \ldots x_{m} \} \subseteq  \{ 0, 1\}^{n}$. For $p \geq 1$, let $D^{(p)}$ denote the $1$-distance
matrix for $(X, d_{p})$. In other words, $D^{(p)} = ( \| x_{i} - x_{j} \|_{p})_{i,j = 0}^{m}$ where $\| \cdot \|_{p}$ denotes
the $\ell_{p}$-norm on $\mathbb{R}^{n}$. As per our observations above, we have
$$D_{p}^{(p)} = (d_{p}(x_{i}, x_{j})^{p})_{i,j = 0}^{m} = (d_{1}(x_{i}, x_{j}))_{i,j = 0}^{m} = D_{1}^{(1)}$$
for all $p \geq 1$. Notice that $D_{1}^{(1)}$ is just $D$ according to the notation of Section \ref{Sec 2}. Hence,
by applying Theorem \ref{Thm 1}, we see that if the set $X$ is affinely dependent, then
$\det (D_{p}^{(p)}) = \det (D_{1}^{(1)}) = 0$ for all $p \geq 1$. So in this case we deduce that $\wp_{(X, d_{p})} = p$
for all $p \geq 1$. In particular, by applying Theorem \ref{Thm 3.5}, it follows that $(X, d_{p})$ does not have
strict $p$-negative type for any $p \geq 1$. Specializing to the case $p = 1$ provides a new proof of one implication
of Murugan's theorem. To establish the converse implication we need to develop two additional results. The first is a
variant of Lemma \ref{Lemma 1}.

\begin{thm}\label{Thm 4}
	Let $\{ x_{0}, x_{1}, \ldots, x_{m} \}$, $m \geq 1$, be a subset of the Hamming
	cube $H_{n}$. Then
	\[ \det\begin{pmatrix} 0 & \1^{T} \\ \1 & D \end{pmatrix}=(-1)^{m-1}2^{m}\det(G),
	\]
        where $D = D_{1}^{(1)} = (d_{1}(x_{i}, x_{j}))_{i,j = 0}^{m}$.
\end{thm}
\begin{proof} Recall that $G$ denotes the Gram matrix $G(x_{1}, \ldots, x_{m}) = (x_{i} \cdot x_{j})_{i,j = 1}^{m}$.
Let $u=(x_{1}\cdot x_{1}, \ldots, x_{m}\cdot x_{m})^{T}$ and
$A=((x_{i}\cdot x_{i})+(x_{j}\cdot x_{j})-2(x_{i}\cdot x_{j}))_{i,j=1}^{m}$. Also, for $k \geq 1$, let
$\1_{k}$ denote the column vector in $\mathbb{R}^{k}$ all of whose entries are $1$. Using (\ref{eqn1}) we have that
\[ \begin{pmatrix} 0 & \1_{m+1}^{T} \\ \1_{m+1} & D \end{pmatrix}=\begin{pmatrix} 0 & 1 & \1_{m}^{T} \\ 1  & 0 & u^{T} \\ \1_{m} & u & A \end{pmatrix}. \]
We proceed by applying elementary row and column operations, as in the proof of Lemma \ref{Lemma 1}. To this end,
let $R_{i}$ and $C_{j}$ denote the $i$-th row and $j$-th column of the above matrix, respectively.
Then for $3\leq i,j\leq m+2$, replace $R_{i}$ by $R_{i}-R_{2}$ and then replace $C_{j}$ by $C_{j}-C_{2}$. This gives that
\[ \det\begin{pmatrix} 0 & \1_{m+1}^{T} \\ \1_{m+1} & D \end{pmatrix}=\det\begin{pmatrix} 0 & 1 & \1_{m}^{T} \\
1  & 0 & u^{T} \\ \1_{m} & u & A \end{pmatrix}=\det\begin{pmatrix} 0 & 1 & \mathbf{0}_{m}^{T} \\ 1  & 0 & u^{T} \\ \mathbf{0}_{m} & u & -2G \end{pmatrix}
	\]
	where here $\mathbf{0}_{m}$ denotes the zero vector in $\mathbb{R}^{m}$. Now, factorizing $-2$ from all of the entries and
        then replacing such a factor in the second row and the second column, we see that
	\[ \det\begin{pmatrix} 0 & \1_{m+1}^{T} \\ \1_{m+1} & D \end{pmatrix}=(-2)^{m}\det\begin{pmatrix} 0 & 1 & \mathbf{0}_{m}^{T} \\
        1  & 0 & u^{T} \\ \mathbf{0}_{m} & u & G \end{pmatrix}.
	\]
	But now we just expand along the top row twice to obtain
	\begin{align*}
	  \det\begin{pmatrix} 0 & \1_{m+1}^{T} \\ \1_{m+1} & D \end{pmatrix}&=(-2)^{m}\det\begin{pmatrix} 0 & 1 & \mathbf{0}_{m}^{T} \\
          1  & 0 & u^{T} \\ \mathbf{0}_{m} & u & G \end{pmatrix}
	\\&=-(-2)^{m}\det\begin{pmatrix} 1 & u^{T} \\ \mathbf{0}_{m} & G \end{pmatrix}
	\\&=-(-2)^{m}\det(G)
	\\&=(-1)^{m-1}2^{m}\det(G),
	\end{align*}
	as required.
\end{proof}

\begin{cor}\label{Cor M}
	Let $\{ x_{0}, x_{1}, \ldots, x_{m} \}$, $m \geq 1$, be a subset of the Hamming
	cube $H_{n}$. If the set of vectors $\{ x_{0}, x_{1}, \dots, x_{m} \}$ is affinely independent, then
	\[ \ip< D^{-1} \1,\1>=\frac{2}{\langle G^{-1}u,u\rangle} > 0,
        \]
         where $D = D_{1}^{(1)} = (d_{1}(x_{i}, x_{j}))_{i,j = 0}^{m}$.
\end{cor}
\begin{proof} Recall that $G$ denotes the Gram matrix $G(x_{1}, \ldots, x_{m}) = (x_{i} \cdot x_{j})_{i,j = 1}^{m}$.
	Since $D$ is invertible, we may apply Lemma \ref{Lemma 2} with $W=0$, $X=\1^{T}$, $Y=\1$ and $Z=D$. This gives that
		\[ \det\begin{pmatrix} 0 & \1^{T} \\ \1 & D \end{pmatrix}=-\det(D)\ip< D^{-1} \1,\1>.
		\]
	So, by Theorems \ref{Thm 2} and \ref{Thm 4}, we see that
	\begin{align*}
	\langle D^{-1}\1,\1\rangle&=-\frac{\det\begin{pmatrix} 0 & \1^{T} \\ \1 & D \end{pmatrix}}{\det(D)}
	\\&=-\frac{(-1)^{m-1}2^{m}\det(G)}{(-1)^{m}2^{m-1}\det(G)\langle G^{-1}u,u\rangle}
	\\&=\frac{2}{\langle G^{-1}u,u\rangle}.
	\end{align*}
        Moreover, the Gram matrix $G$ is positive definite because the vectors $x_{1},\dots,x_{m}$ are linearly independent.
	Hence $G^{-1}$ is positive definite, and so $\langle G^{-1}u,u\rangle > 0$.
\end{proof}

Now if the set of vectors $X = \{ x_{0}, \ldots, x_{m} \} \subseteq H_{n}$ is affinely independent, then
$\det (D) \not= 0$ by Theorem \ref{Thm 2}. Moreover, $ \ip< D^{-1} \1,\1> > 0$ by Corollary \ref{Cor M}.
So we see that $\wp_{(X, d_{1})} > 1$ by (\ref{eqn4}) and Theorem \ref{Thm 3.5}.
In summary, we have obtained the following version of Murugan's theorem. 

\begin{thm}\label{Murugan}
  Let $X = \{x_0,x_1,\dots,x_m\}$, $m \geq 1$, be a subset of $\{ 0, 1 \}^{n}$. Then $\wp_{(X,d_p)} \ge p$ for all $p \geq 1$.
  Furthermore, the following conditions are equivalent.
\begin{enumerate}
  \item $X$ is affinely independent.
  \item $(X, d_1)$ has strict $1$-negative type.
  \item $\wp_{(X,d_p)} > p$ for some $p \ge 1$.
  \item $\wp_{(X,d_p)} > p$ for all $p \ge 1$.
\end{enumerate}
\end{thm}

\section{Affinely independent subsets of the Hamming cube}\label{Sec 4}

In the proof of Murugan's result given in the previous section, it was shown that if $\{x_{0},x_{1},\dots,x_{m}\}$ is an
affinely independent subset of $H_{n}$, then $\langle D^{-1}\1,\1\rangle > 0$
(where $D = D_{1}^{(1)} = (d_{1}(x_{i}, x_{j}))_{i,j = 0}^{m}$). However, there is actually more that can
be said in this setting. As stated earlier, results of Hjorth et al.\ \cite{Hj1} and Murugan \cite{MM} imply
that any unweighted metric tree $T$ on $n+1$ vertices embeds isometrically into $H_{n}$ as an affinely
independent set. For such embedded trees we may compute the precise value of $\langle D^{-1}\1,\1\rangle$.
In fact, just as Graham and Pollak \cite{GP} showed that $\det(D)$ does not depend upon the geometry of the particular tree,
we now show that this is also the case for the positive quantity $\langle D^{-1}\1,\1\rangle$.

\begin{thm}\label{tree}
Let $D$ be the distance matrix of an unweighted metric tree on $n+1$ vertices. Then
\[ \ip< D^{-1} \1,\1> = \frac{2}{n}. \]
\end{thm}

\begin{proof}
  Denote the vertices of the tree by $v_{0},\dots,v_{n}$ and for each $i$, $0\leq i\leq n$, let $\delta_{i}$ be the degree of the vertex $v_{i}$.
  Let $A=(a_{ij})_{i,j=0}^{n}$ be the adjacency matrix of the tree and write $D^{-1}=(d_{ij}^{\ast})_{i,j=0}^{n}$. By
  Graham and Lov{\'a}sz \cite[Lemma 1]{GL},
	\[ d_{ii}^{\ast}=\frac{(2-\delta_{i})(2-\delta_{i})}{2n}-\frac{\delta_{i}}{2}
	\]
	for all $0\leq i\leq n$, and
	\[d_{ij}^{\ast}=\frac{(2-\delta_{i})(2-\delta_{j})}{2n}+\frac{a_{ij}}{2}
	\]
	for all $0\leq i,j\leq n$ such that $i\neq j$.  We then compute that
	\begin{align*}
	\langle D^{-1}\1,\1\rangle&=\sum_{i,j=0}^{n}d_{ij}^{\ast}
	\\&=\sum_{\substack{i,j=0 \\ i\neq j}}^{n}d_{ij}^{\ast}+\sum_{i=0}^{n}d_{ii}^{\ast}
	\\&=\frac{1}{2n}\sum_{i,j=0}^{n}(2-\delta_{i})(2-\delta_{j})+\frac{1}{2}\bigg(\sum_{i,j=0}^{n}a_{ij}-\sum_{i=0}^{n}\delta_{i}\bigg)
	\\&=\frac{1}{2n}\sum_{i,j=0}^{n}(2-\delta_{i})(2-\delta_{j}).
	\end{align*}
	Now we use the fact that the sum of the degrees of the vertices of a tree on $n+1$ vertices is $2n$. Consequently,
	\begin{align*}
	\langle D^{-1}\1,\1\rangle&=\frac{1}{2n}\sum_{i,j=0}^{n}(2-\delta_{i})(2-\delta_{j})
	\\&=\frac{1}{2n}\bigg(\sum_{i=0}^{n}2-\delta_{i}\bigg)^{2}
	\\&=\frac{1}{2n}\bigg(2(n+1)-2n\bigg)^{2}
	\\&=\frac{2}{n},
	\end{align*}
	as required.
\end{proof}

The condition $\ip< D^{-1} \1,\1> = 2/n$ in Theorem \ref{tree} is not unique to embedded $n + 1$ point trees
in $H_{n}$. If, for example, we set $x_{0} = \mathbf{0}$, $x_{1} = (1,0,0)$, $x_{2} = (0,1,0)$ and $x_{3} = (1,1,1)$ in $H_{3}$, then it is
easy to verify that $\ip< D^{-1} \1,\1> = 2/3$. However, $\{ x_{0}, x_{1}, x_{2}, x_{3} \}$ is certainly not an embedded tree
in $H_{3}$.

As for what can be said about affinely independent subsets of $H_{n}$ that may not have the structure of an unweighted metric tree,
we have the following conjecture.

\begin{con*}
	Let $\{ x_{0}, x_{1}, \ldots, x_{m} \}$, $m \geq 1$, be a  subset of the Hamming cube $H_{n}$.
        If the set of vectors $\{ x_{0}, x_{1}, \ldots, x_{m} \}$ is affinely independent, then $\ip<D^{-1} \1,\1> \ge 2/n$.
\end{con*}

We have confirmed this conjecture using a computer algebra package for all integers $m,n \leq 5$. In addition, tens of thousands
of random tests in larger dimensions has not provided any counterexamples to date. No clear arithmetic reason for the conjectured
lower bound has come to our attention. Notably, the denominators of the entries of $D^{-1}$ can be large compared to $n$. It is
fascinating to ask what, if any, geometric information is encoded by the quantity $\ip< D^{-1} \1,\1>$ in this context.

\section*{Acknowledgements}

The work of the second and third authors was supported by the Research Training Program of the
Department of Education and Training of the Australian Government.

\bibliographystyle{amsalpha}

\end{document}